\documentclass[11pt]{article}
\usepackage{amsmath,amssymb,epsfig}

\usepackage{algorithm}
\usepackage{algpseudocode}
\usepackage{algorithmicx}

\usepackage{geometry}
\usepackage{url}
\usepackage{array}

\usepackage{layout}
\usepackage{color}
\usepackage{gauss}

\usepackage{graphicx}
\usepackage{epsfig}
\usepackage{subfig}
\usepackage{grffile}

\usepackage{tikz}
\usepackage[vlined,lines numbered,ruled,algo2e]{algorithm2e}

\usepackage{amsthm}
\newtheorem{theorem}{Theorem}[section]
\newtheorem{lemma}[theorem]{Lemma}

\newtheorem{corollary}[theorem]{Corollary}

\newtheorem{example}[theorem]{Example}

\newtheorem{assumption}[theorem]{Assumption}

\newcommand{\C}{{\mathbb C}}

\newcommand{\calK}{{\mathcal K}}
\newcommand{\calW}{{\mathcal W}}

\title{A Krylov subspace method for the approximation of \\ bivariate matrix functions}

\author{Daniel Kressner}
\date{\today}

\begin{document}
\maketitle
\begin{abstract}
Bivariate matrix functions provide a unified framework for various tasks in numerical linear algebra, including the solution of linear matrix equations and the application of the Fr\'echet derivative. In this work, we propose a novel tensorized Krylov subspace method for approximating such bivariate matrix functions and analyze its convergence. While this method is already known for some instances, our analysis appears to result in new convergence estimates and insights for all but one instance, Sylvester matrix equations. 
\end{abstract}

\section{Introduction}

Given a univariate function $f(z)$ defined in the neighborhood of the spectrum $\Lambda(A)$ of a matrix $A\in \C^{n\times n}$, the numerical computation of the \emph{matrix function} $f(A) \in \C^{n\times n}$ has been studied intensively during the last decades; see~\cite{Frommer2008a,Guettel2013,Higham2008} for surveys.
%For example, the matrix inverse corresponds to $f(z) = 1/z$ and the matrix exponential corresponds to $f(z) = \exp(z)$.
The extension of the notion of matrix functions to bivariate or, more generally, multivariate functions $f$  has a long history as well, notably in the context of holomorphic functional calculus and operator theory; see~\cite[Sec. 3]{Kressner2014bivariate} for a detailed discussion and references. In the numerical analysis literature, however, bivariate matrix functions have been discussed mostly for special cases only. 

Given two matrices $A\in \C^{m\times m}$ and $B\in \C^{n\times n}$ and a bivariate function $f(x,y)$ defined in a neighourhood of $\Lambda(A)\times\Lambda(B)$, the \emph{bivariate matrix function} $f\{A,B\}$ is a linear operator on $\C^{m\times n}$. We will recall the formal definition of $f\{A,B\}$ in Section~\ref{sec:prelims} below. Linear matrix equations and Fr\'echet derivatives constitute the most widely known instances of bivariate matrix functions:
\begin{enumerate}
 \item  For $f(x,y) = 1/(x+y)$ the matrix $X = f\{A,B\}(C)$ is the solution of the Sylvester matrix equation
\begin{equation} \label{eq:sylvester}
  AX + XB^T = C, \qquad C \in \C^{m\times n},
\end{equation}
where $B^T$ denotes the complex transpose of $B$ and $C$ is often of low rank.
When $B$ equals $\bar A$ and $C$ is Hermitian,~\eqref{eq:sylvester} is called Lyapunov matrix equation. Such matrix equations play an important role in control, e.g, for computing the Gramians in balanced truncation model reduction of linear time-invariant control systems. They also arise from structured discretizations of partial differential equations. See~\cite{Benner2013,Simoncini2016} for references.

\item There are several variants of~\eqref{eq:sylvester} that fit the framework of bivariate matrix functions. The solution of the Stein equation $AXB^T -X = C$ is given by $X = f\{A,B\}(C)$ with $f(x,y) = 1/(1-xy)$. More generally, for $f(x,y) = 1/p(x,y)$, with a bivariate polynomial $p(x,y) = \sum_{i = 0}^k \sum_{j = 0}^\ell p_{ij} x^i y^j$, the matrix $X = f\{A,B\}(C)$ is the solution of the matrix equation
\[
 \sum_{i = 0}^k \sum_{j = 0}^\ell p_{ij} A^i X (B^T)^j = C,
\]
which has been considered, e.g., in~\cite{Dalecki1974,Lancaster1970}. 

Time-limited and frequency-limited balanced truncation model reduction~\cite{Benner2016limitedgramians,Gawronski1990} give rise to matrix equations that involve matrix exponentials and logarithms. For example, the reachability Gramian corresponding to a time interval $0\le t_s < t_e \le \infty$ satisfies an equation of the form
\begin{equation} \label{eq:timelimitedeqn}
 AX + XA^* = -\exp(t_sA)C \exp(t_sA^*) + \exp(t_eA)C \exp(t_eA^*),
\end{equation}
where again $C$ is often of low rank. The solution of~\eqref{eq:timelimitedeqn} can be expressed as $X = f\{A,\bar A\}(C)$ with
\begin{equation} \label{eq:timelimited}
 f(x,y) = \frac{\exp(t_e(x+y))-\exp(t_s(x+y))}{x+y}.
\end{equation} In the analogous situation for frequency-limited balanced truncation, the corresponding function takes the form
\begin{equation} \label{eq:frequencylimited}
  f(x,y) = -\frac{g(x)+g(y)}{x+y},\quad g(z) = \mathsf{Re}\Big( \frac{\mathrm{i}}{\pi}\ln\Big(\frac{z+\mathrm{i}\omega_2}{z+\mathrm{i}\omega_1} \Big)  \Big), \quad 0\le\omega_1<\omega_2\le\infty,
\end{equation}
where $\mathsf{Re}$ denotes the real part of a complex number.

\item Given a (univariate) matrix function $f(A)$ and the finite difference quotient
\begin{equation}
\label{eq:finitedifference}
f^{[1]}(x,y):=f[x,y] = \left\{
\begin{array}{ll}
 \frac{f(x)-f(y)}{x-y}, & \text{for } x\not=y, \\
 f^\prime(x), & \text{for } x=y, \\
\end{array}
\right.                                          
\end{equation}
the matrix $X = f^{[1]}\{A,A^T\}(C)$ is the Fr\'echet derivative of $f$ at $A$ in direction $C$; see~\cite[Thm. 5.1]{Kressner2014bivariate}.
\end{enumerate}

In this work, we consider the numerical approximation of $f\{A,B\}(C)$ for large matrices $A$ and $B$. As the size of the involved matrices grows, it becomes necessary to impose additional structure before attempting this task. We assume that matrix-vector multiplications with $A$ and $B$ are feasible because, for example, $A$ and $B$ are sparse. Moreover, $C$ is assumed to have low rank. The latter is a common assumption in numerical solvers for large-scale matrix equations~\eqref{eq:sylvester}, but we also refer to~\cite{Grasedyck2003a,Haber2016,KressnerMasseiRobol2017,Palitta2017} for works that consider other types of data-sparsity for $C$.

Given a rank-one matrix $C = cd^T$, the method proposed in this paper makes use of the two Krylov subspaces generated by the matrices $A,B$ with starting vectors $c,d$. An approximation to $f\{A,B\}(C)$ is then selected from the tensor product of these two subspaces.
Our method already exists for several of the instances mentioned above. For $f(x,y) = 1/(x+y)$, it corresponds to a widely known Krylov subspace method for Lyapunov and Sylvester equations~\cite{Jaimoukha1995,Saad1990}. For the functions~\eqref{eq:timelimited} and~\eqref{eq:frequencylimited}, our method corresponds to the Krylov subspace methods presented in~\cite{Kurschner2017} and~\cite{Benner2016limitedgramians}, respectively.
For the Fr\'echet derivative, the algorithm presented in this paper has been proposed independently in~\cite{Kandolfprep}. For Lyapunov and Sylvester equations, the convergence of these methods has been analyzed in detail; see, e.g.,~\cite{Beckermann2011,Simoncini2009a}. For all other instances, the general theory presented in this work appear to result in previously unknown convergence estimates.

We note in passing that the algorithm proposed in this paper shares similarities with a recently proposed Krylov subspace method for performing low-rank updates of matrix functions~\cite{BeckermannKressnerSchweitzer2017}. However, unlike Fr\'echet derivatives, it does not seem to be possible to capture low-rank updates within the presented framework for bivariate matrix functions.

\section{Preliminaries} \label{sec:prelims}

We first recall the definition of bivariate matrix functions and their basic properties from~\cite{Kressner2014bivariate}.
Let $\Pi_{k,\ell}$ denote the set of all bivariate polynomials of degree at most $(k,\ell)$, that is, for $p \in \Pi_{k,\ell}$  we have that 
$p(x,y)$ has degree at most $k$ in $x$ and degree at most $\ell$ in $y$.
Every such polynomial takes the form
\[
 p(x,y) = \sum_{i = 0}^k \sum_{j = 0}^\ell p_{ij} x^i y^j, \qquad p_{ij} \in \C.
\]
The bivariate matrix function corresponding to $p$ and evaluated at $A \in \C^{m\times m}$, $B\in \C^{n\times n}$ is defined as
\begin{equation} \label{eq:polynomialform}
  p\{A,B\}: \C^{m\times n} \to \C^{m\times n}, \qquad p\{A,B\}(C):=\sum_{i = 0}^k \sum_{j = 0}^\ell  A^i C (B^T)^j.
\end{equation}
This definition extends via Hermite interpolation to general functions $f$ that are sufficiently often differentiable at the eigenvalues of $A$ and $B$; see~\cite[Def. 2.3]{Kressner2014bivariate} for details. A more compact and direct definition is possible when $f$ is analytic.
\begin{assumption} \label{assumption:holo}
There exist domains $\Omega_A,\Omega_B \subset \C$ containing the eigenvalues of $A$ and $B$, respectively, such that 
$f_y(x):=f(x,y)$ is analytic in $\Omega_A$ for every $y \in \Omega_B$ and  $f_x(y):=f(x,y)$ is analytic in $\Omega_B$ for every $x \in \Omega_A$. 
\end{assumption}
\noindent By Hartog's theorem~\cite{Krantz1982}, Assumption~\ref{assumption:holo} implies that $f$ is analytic in $\Omega_A\times \Omega_B$. Moreover, we have \begin{equation} \label{eq:contour}
 f\{A,B\}(C) = -\frac{1}{4\pi^2} \oint_{\Gamma_A} \oint_{\Gamma_B} f(x,y) (xI-A)^{-1} C (yI-B^T)^{-1}\,\text{d}y\,\text{d}x,
\end{equation}
where $\Gamma_A\subset \Omega_A$ and $\Gamma_B \subset \Omega_B$ are closed contours enclosing the eigenvalues of $A$ and $B$, respectively.

Diagonalizing one of the two matrices $A,B$ relates bivariate matrix functions to (univariate) matrix functions of the other matrix. A similar result has already been presented in~\cite[Sec. 6]{Kressner2014bivariate}; we include its proof for the sake of completeness.
\begin{lemma} \label{lemma:diagB} Suppose that Assumption~\ref{assumption:holo} holds and that there is an an invertible matrix $Q$ such that $Q^{-1} B Q  = \mathrm{diag}(\mu_1,\ldots,\mu_n)$. Then
 \[
  f\{A,B\}(C) = \begin{bmatrix}
                 f_{\mu_1}(A) \tilde c_1 & f_{\mu_2}(A) \tilde c_2 & \cdots & f_{\mu_n}(A) \tilde c_n
                \end{bmatrix} Q^T,
 \]
 with $C Q^{-T} =: \tilde C = \begin{bmatrix}
                   \tilde c_1 & \cdots & \tilde c_n
                  \end{bmatrix}
 $ and $f_\mu:=f(x,\mu)$.
\end{lemma}
\begin{proof}
 Setting $\Lambda_B = \mathrm{diag}(\mu_1,\ldots,\mu_n)$, we obtain from~\eqref{eq:contour} that
 \begin{eqnarray*}
  f\{A,B\}(C) &=& -\frac{1}{4\pi^2} \oint_{\Gamma_A}  (xI-A)^{-1} \tilde C \Big[ \oint_{\Gamma_B} f(x,y) (yI-\Lambda_B)^{-1}\,\text{d}y \Big] Q^T\,\text{d}x \\
  &=& \frac{1}{2\pi \mathrm{i}} \oint_{\Gamma_A}  (xI-A)^{-1} \tilde C \cdot \mathrm{diag}(f_{\mu_1}(x),\ldots, f_{\mu_n}(x) )Q^T\, \text{d}x \\
  &=& \frac{1}{2\pi \mathrm{i}} \oint_{\Gamma_A} \begin{bmatrix}
                  f_{\mu_1}(x) (xI-A)^{-1} \tilde c_1 &  \cdots &  f_{\mu_n}(x) (xI-A)^{-1} \tilde c_n
                \end{bmatrix} Q^T\, \text{d}x,
 \end{eqnarray*}
 which concludes the proof, using the contour integral representation of $f_{\mu}(A)$.
\end{proof}
\noindent For the case $f(x,y) = 1/(x+y)$, the result of Lemma~\ref{lemma:diagB} is related to algorithms for Sylvester equation with large $m$ but relatively small $n$; see~\cite{Simoncini1996}.

If both $A,B$ are diagonalizable, that is, additionally to the assumption of Lemma~\ref{lemma:diagB} there exists an invertible matrix $P$ such that $P^{-1} A P  = \text{diag}(\lambda_1,\ldots,\lambda_m)$ then the result of the lemma implies 
\[
 f\{A,B\}(C) = P \left( \begin{bmatrix}
                         f(\lambda_1,\mu_1) & \cdots & f(\lambda_1,\mu_n) \\
                         \vdots & & \vdots \\
                         f(\lambda_m,\mu_1) & \cdots & f(\lambda_m,\mu_n)
                        \end{bmatrix} \circ C
 \right)Q^T, \qquad \tilde C:= P^{-1}C Q^{-T},
\]
where $\circ$ denotes the elementwise (or Hadamard) product.

\section{Algorithm} 

For the sake of simplifying the presentation, we assume that $C$ has rank $1$ and can thus be written as $C = c d^T$ for nonzero vectors $c, d\in \C^n$. We comment on the extension to (small) ranks larger than $1$ below.

Our method proceeds by constructing orthonormal bases for the Krylov subspaces
\[
\calK_k(A,b) = \text{span}\big\{c,Ac,\ldots,A^{k-1}c\big\}, \qquad 
\calK_\ell(B,d) = \text{span}\big\{d,Bd,\ldots,B^{\ell-1}d\big\},
\]
When $k\le m$ and $\ell \le n$, these subspaces are generically of dimension $k$ and $\ell$, which will be assumed in the following.
The Arnoldi method~\cite{Stewart2001} applied to $\calK_k(A,b)$, $\calK_\ell(B,d)$ not only produces orthonormal bases $U_k \in \C^{m\times k}$, $V_\ell \in \C^{n\times \ell}$
but also yields Arnoldi decompositions
\begin{eqnarray}
 AU_k &=& U_k G_k + g_{k+1,k} u_{k+1} e_k^T, \label{eq:arnoldia} \\ 
 BV_\ell &=& V_\ell H_\ell + h_{\ell+1,\ell} v_{\ell+1} e_\ell^T, \label{eq:arnoldib}
\end{eqnarray}
where $G_k = U_k^\ast A U_k$ and $H_\ell = V_\ell^\ast B V_\ell$ are upper Hessenberg matrices, $e_k$ and $e_\ell$ denote the $k$th and $\ell$th unit vectors of suitable length, $g_{k+1,k}$ and $h_{\ell+1,\ell}$ are complex scalars. If $k<m$ and $\ell < n$ then $[U_k,u_{k+1}]$ and $[V_\ell,u_{\ell+1}]$ form orthonormal bases of 
$\calK_{k+1}(A,b)$ and $\calK_{\ell+1}(B,d)$, respectively.

We search for an approximation to $f\{A,B\}(C)$ in $\calK_k(A,b) \times \calK_\ell(B,d)$. Every such approximation takes the form $U_k X_{k,\ell} V_\ell^T$ with some matrix $X_{k,\ell} \in \C^{k\times \ell}$. For reasons that become clear in Section~\ref{sec:convergence} below, a suitable (but possibly not the only) choice for this matrix is obtained by evaluating the compressed  function:
\[
X_{k,\ell} = f\big\{ U_k^* A U_k, V_k^* B V_k \big\}(U_k^* C \overline{V}_\ell) =
f\{ G_k, H_\ell \}(\tilde c \tilde d^T), 
\]
with $\tilde c = U_k^* c$, $\tilde d = V_\ell^* d$.

\begin{algorithm}
\caption{Arnoldi method for approximating $f\{A,B\}(C)$ with $C = cd^T$}
\label{alg:arnoldi}
\begin{algorithmic}[1]
\State Perform $k$ steps of the Arnoldi method to compute an orthonormal basis $U_k$ of $\calK_k(A,c)$ and $G_k = U_k^\ast A U_k$, $\tilde c = U_k^\ast c$.
\State Perform $\ell$ steps of the Arnoldi method to compute an orthonormal basis $V_\ell$ of $\calK_\ell(B,d)$ and $H_\ell = V_\ell^\ast B V_\ell$, $\tilde d = V_k^\ast d$.
\State\label{line:compressed}%
Compute bivariate matrix function
$
  X_{k,\ell} = f\{G_k,H_\ell\}( \tilde c \tilde d^T ).
$
\State Return $U_k X_{k,\ell} V_\ell^T$.
\end{algorithmic}
\end{algorithm}

The described procedure is summarized in Algorithm~\ref{alg:arnoldi}.
We conclude this section with several remarks:
\begin{enumerate}
 \item For the compressed function in Line~\ref{line:compressed}, one requires that $f$ is defined on $\Lambda(G_k) \times \Lambda(H_\ell)$. Considering the numerical ranges
%\begin{equation} \label{eq:numerange}
\[
\calW(A) = \big\{ w^* A w:\, w\in \C^m,\,\|w\|_2 = 1\big\}, \quad 
\calW(B) = \big\{ w^* B w:\, w\in \C^m,\,\|w\|_2 = 1\big\},
%\end{equation}
\]
the following assumption guarantees that this requirement is met; it is also needed in the convergence analysis of Section~\ref{sec:convergence}.
 \begin{assumption} \label{assumption:holnr} Assumption~\ref{assumption:holo} is satisfied with domains $\Omega_A,\Omega_B$ satisfying $\calW(A)\subset \Omega_A$ and $\calW(B)\subset \Omega_B$.
\end{assumption}
Because of $\Lambda(G_k)\subset \calW(G_k) \subset \calW(A)$ and 
$\Lambda(H_k)\subset \calW(H_k) \subset \calW(B)$, Assumption~\ref{assumption:holnr} implies that $f\{G_k,H_\ell\}$ is well defined.

General-purpose approaches to evaluating the small and dense bivariate matrix function $f\{G_k,H_\ell\}( \tilde c \tilde d^T )$ in Line~\ref{line:compressed} are discussed in~\cite[Sec. 6]{Kressner2014bivariate}. However, let us stress that it is generally advisable to use an approach that is tailored to the function $f$ at hand. For example, for $f(x,y) = 1/(x+y)$ this amounts to solving a small linear matrix equation, for which the Bartels-Stewart algorithm~\cite{Bartels1972} should be used. For the finite difference quotient~\eqref{eq:finitedifference}, a suitable method is discussed in Section~\ref{sec:frechet} below.
 \item As in the case of univariate functions, there is no reliable stopping criterion for general $f$  that would allow to choose $k$, $\ell$ such that Algorithm~\ref{alg:arnoldi} is guaranteed to return an approximation with a prescribed accuracy. In the spirit of existing heuristic criteria, we propose to use the approximation
 \[
  \|f\{A,B\}(cd^T) - U_k X_{k,\ell} V_\ell^T\|_F \approx 
  \|U_{k+h} X_{k+h,\ell+h} V_{\ell+h}^T - U_k X_{k,\ell} V_\ell^T\|_F:=e_{k,\ell,h}
 \]
 for some small integer $h$, say $h = 2$. As already explained in, e.g.,~\cite[Sec. 2.3]{BeckermannKressnerSchweitzer2017}, the quantity $e_{k,\ell,h}$ is inexpensive to check because
 \[
  e_{k,\ell,h} = \left\| U_{k+h}\left( X_{k+h,\ell+h} - \begin{bmatrix}X_{k,\ell} & 0 \\ 0 & 0 \end{bmatrix}  \right) V_{\ell+h}^T  \right\|_F = 
  \left\| X_{k+h,\ell+h} - \begin{bmatrix}X_{k,\ell} & 0 \\ 0 & 0 \end{bmatrix} \right\|_F.
 \]
 If $e_{k,\ell,d}$ is smaller than a user-specified tolerance, the output of
 Algorithm~\ref{alg:arnoldi} is accepted. Otherwise, $k$ and $\ell$ are increased, the orthonormal bases $U_k, V_\ell$ are extended and Step 3 is repeated. It may be desirable to increase $k$ and $\ell$ separately. For example, one could increase $k$ if \[\left\| X_{k+h,\ell} - \begin{bmatrix}X_{k,\ell} \\ 0 \end{bmatrix} \right\|_F \ge \left\| X_{k,\ell+h} - \begin{bmatrix}X_{k,\ell} & 0 \end{bmatrix} \right\|_F\]
 and increase $\ell$ otherwise.
 
  Again, we emphasize that better stopping criteria may exist for specific choices of $f$. This is particularly true for linear matrix equations; see~\cite{Palitta2018} and the references therein.
 \item Algorithm~\ref{alg:arnoldi} extends to matrices $C$ of rank $r>1$ by replacing the Arnoldi method in Steps 1 and 2 by a block Arnoldi method, by a global Arnoldi method, or by splitting $C$ into $r$ rank-1 terms; see~\cite{Frommer2017} for a comparison of these approaches in a related setting.
 
\end{enumerate}

\section{Exactness properties and convergence analysis} \label{sec:convergence}

In this section, we analyze the convergence of Algorithm~\ref{alg:arnoldi} following a strategy commonly used for matrix functions; see, in particular,~\cite{BeckermannKressnerSchweitzer2017}. First, we establish that Algorithm~\ref{alg:arnoldi} is exact (that is, it returns $f\{A,B\}(cd^T)$) for polynomials of bounded degree. This then allows us to relate its error for general functions to a bivariate polynomial approximation problem on the numerical ranges. 
\begin{lemma} \label{lemma:exact}
Algorithm~\ref{alg:arnoldi} is exact if $f \in \Pi_{(k-1,\ell-1)}$.
\end{lemma}
\begin{proof}
The following well-known exactness property of the Arnoldi method (see, e.g.,~\cite{Saad1992}) follows by induction from~\eqref{eq:arnoldia}--\eqref{eq:arnoldib}: \[
 A^i c = U_k \big( G_k \big)^i U_k^\ast c,\quad i = 0,\ldots,k-1,  \qquad 
 B^j d = V_\ell \big( H_\ell \big)^j V_\ell^\ast d,\quad  j = 0,\ldots,\ell-1.
\]
By writing $f(x,y) = \sum_{i = 0}^{k-1} \sum_{j = 0}^{\ell-1} p_{ij} x^i y^j$ and using~\eqref{eq:polynomialform}, this gives
\begin{eqnarray*}
 f\{A,B\}(cd^T) &= & \sum_{i = 0}^{k-1} \sum_{j = 0}^{\ell-1} A^i c (B^j d)^T = U_k\Big( \sum_{i = 0}^{k-1} \sum_{j = 0}^{\ell-1} G_k^i U_k^\ast c d^T\,\overline V_\ell (H_\ell^T)^j \Big) V_\ell^T \\&=& U_k\cdot f\{G_k,H_\ell\}\big( U_k^\ast c d^T\, \overline V_\ell \big) \cdot  V_\ell^T,
\end{eqnarray*}
which corresponds to what is returned by Algorithm~\ref{alg:arnoldi}.
\end{proof}

To treat general functions, we will need to estimate the norm of $f\{A,B\}$ induced by the Frobenius norm on $\C^{m\times n}$:
%\begin{equation} \label{eq:fnorm}
\[
 \|f\{A,B\}\|:=\max\big\{ \|f\{A,B\}(C)\|_F:\,C\in \C^{m\times n},\,\|C\|_F = 1 \big\}.
 \]
%\end{equation}
For a (univariate) matrix function $f(A)$, the seminal result by Crouzeix and Palencia~\cite{CrouzeixPalencia2017} states that
$\|f(A)\|_2\le (1+\sqrt{2})\max_{x\in \calW(A)} |f(x)|$. Theorem 1.1 in~\cite{Gil2011} appears to be the only result in the literature that aims at establishing norm bounds for general bivariate functions. This result provides an upper bound in terms of Henrici's departure from normality for $A$ and $B$~\cite{Henrici1962} as well as the maximal absolute value of $f$ and its derivatives on convex hulls of $\Lambda(A),\Lambda(B)$. The following lemma provides an upper bound in terms of the maximal absolute value of $f$ on the numerical ranges, which is better suited for our purposes.
\begin{lemma}  \label{lemma:normestimate} Suppose that Assumption~\ref{assumption:holnr} holds and let 
$\mathbb E_A, \mathbb E_B$ be compact connected sets such that
$\calW(A)\subset \mathbb E_A \subset \Omega_A$ and
$\calW(B)\subset \mathbb E_B \subset \Omega_B$.  Let $\mathrm{len}(\partial \mathbb E_A)$ denote the length of the boundary curve $\partial \mathbb E_A$ of $\mathbb E_A$, let $d_A(\cdot)$ denote the distance between a subset of $\mathbb C$ and $\calW(A)$, and define analogous quantities for $B$.
Then
\[\|f\{A,B\}\| \le M \cdot \max_{x \in \mathbb E_A, y\in \mathbb E_B} |f(x,y)|,\] where 
\begin{enumerate}
 \item[(a)] $M = 1$ if both $A$ and $B$ are normal;
 \item[(b)] $M = 1+\sqrt{2}$ if $A$ or $B$ are normal;
 \item[(c)] $M = \frac{1+\sqrt{2}}{2\pi} \min\big\{\frac{\mathrm{len}(\partial \mathbb E_A)}{d_A(\partial \mathbb E_A)},\frac{\mathrm{len}(\partial \mathbb E_B)}{d_B(\partial \mathbb E_B)}\big\}$ otherwise, under the additional assumption that $d_A(\partial \mathbb E_A)>0$ or $d_B(\partial \mathbb E_B)>0$.
\end{enumerate}
\end{lemma}
\begin{proof} \emph{(a) and (b).} Assume that $B$ is normal. The result of Lemma~\ref{lemma:diagB}, with $Q$ chosen unitary, implies
\begin{eqnarray*}
 \|f\{A,B\}(C)\|^2_F &=& \sum_{j = 1}^m \| f_{\mu_j}(A) \tilde c_j\|_2^2 \le \sum_{j = 1}^m \| f_{\mu_j}(A) \|^2_2 \|\tilde c_j\|_2^2 \\
 &=& M^2 \sum_{j = 1}^m \max_{x\in \mathbb E_A} |f_{\mu_j}(x)|^2 \cdot \|\tilde c_j\|_2^2 \le M^2 \max_{x\in \mathbb E_A,y\in \mathbb E_B} |f(x,y)|^2 \cdot \|C\|_F^2,
\end{eqnarray*}
with $M = 1$ if $A$ is also normal and $M = 1+\sqrt{2}$ otherwise~\cite{CrouzeixPalencia2017}. The proof is analogous when $B$ is normal and $A$ is not.

\emph{(c).} Starting from the representation~\eqref{eq:contour}, we obtain
\begin{eqnarray*}
  f\{A,B\}(C) &=& -\frac{1}{4\pi^2} \oint_{\partial \mathbb E_B} \Big[ \oint_{\partial \mathbb E_A} f(x,y) (xI-A)^{-1} \text{d}x \Big] C (yI-B^T)^{-1}\,\text{d}y \\
  &=& \frac{1}{2\pi \mathrm{i}} \oint_{\partial \mathbb E_B} f_y(A) C (yI-B^T)^{-1}\,\text{d}y
\end{eqnarray*}
and, in turn,
\[
 \|f\{A,B\}(C)\|_F \le \frac{1}{2\pi} \max_{y \in \mathbb E_B} \|f_y(A)\|_2 \|C\|_F \cdot \oint_{\partial \mathbb E_B} \big\| (yI-B^T)^{-1} \big\|_2\,\text{d}y
\]
Combined with $\| (yI-B^T)^{-1} \|_2 \le 1/d_B(y)$,  this shows
$
 \|f\{A,B\}\| \le  \frac{1+\sqrt{2}}{2\pi} \frac{\mathrm{len}(\partial \mathbb E_B)}{d_B(\partial \mathbb E_B)}.
$
Analogously, one establishes the same inequality with $B$ replaced by $A$.
\end{proof}

It remains an open and interesting problem to study whether the constant in Lemma~\ref{lemma:normestimate} (c) can be replaced by, say, $M = (1+\sqrt{2})^2$.
\begin{theorem} \label{theorem:convergence}
Let $\mathbb E_A$, $\mathbb E_B$, and $M$ be defined as in Lemma~\ref{lemma:normestimate} and suppose that the assumptions of the lemma hold. Then the output of Algorithm~\ref{alg:arnoldi} satisfies the error bound
\[
 \|f\{A,B\}(cd^T) - U_k X_{k,\ell} V_\ell^T \|_F \le 2M \|c\|_2 \|d\|_2 \cdot \inf_{p \in \Pi_{k-1,\ell-1}} \max_{x \in \mathbb E_A, y\in \mathbb E_B} |f(x,y)-p(x,y)|.
\]
\end{theorem}
\begin{proof}
 Let $p \in \Pi_{k-1,\ell-1}$. By Lemma~\ref{lemma:exact}, we have
 \[
  p\{A,B\}(cd^T) = U_k\cdot p\{G_k,H_\ell\}\big( \tilde c \tilde d^T \big)\cdot V_\ell^T, \quad \tilde c = U_k^\ast c, \quad \tilde d = V_\ell^\ast c.
 \]
 Hence,
 \begin{eqnarray}
  && f\{A,B\}(cd^T) - U_k X_{k,\ell} V_\ell^T \nonumber \\
  &=& f\{A,B\}(cd^T) - p\{A,B\}(cd^T) - U_k \Big( f\{G_k,H_\ell\}\big( \tilde c \tilde d^T \big) - p\{G_k,H_\ell\}\big( \tilde c \tilde d^T \big) \Big) V_\ell^T \nonumber \\
  &=& e\{A,B\}(cd^T) - U_k\cdot e\{G_k,H_\ell\}\big( \tilde c \tilde d^T \big)\cdot V_\ell^T \label{eq:expression}
 \end{eqnarray}
 with $e = f-p$. Applying Lemma~\ref{lemma:normestimate} and using that the numerical ranges of $G_k$ and $H_k$ are contained in $A$ and $B$, respectively, we have
 \[
  \max\{ \|e\{A,B\}\|_F, \|e\{G_k,H_\ell\}\|_F \} \le M\cdot \max_{x \in \mathbb E_A, y\in \mathbb E_B} |e(x,y)|
 \]
 Inserted into~\eqref{eq:expression}, this gives
 \[
  \|f\{A,B\}(cd^T) - U_k X_{k,\ell} V_\ell^T\|_F \le 2M \|c\|_2 \|d\|_2 \cdot \max_{x \in \mathbb E_A, y\in \mathbb E_B} |e(x,y)|,
 \]
 Because $p$ was chosen arbitrary, the result of the theorem follows.
\end{proof}

Combining Lemma~\ref{lemma:normestimate} with existing results on polynomial multivariate approximation yields concrete convergence estimates.  For example, let us consider the case of Hermitian matrices $A$ and $B$. By a suitable reparametrization, we may assume without loss of generality that $\calW(A)=\calW(B)=[-1,1]$.  By Assumption~\ref{assumption:holnr}, there is $\rho>1$ such that $f$ is analytic on $E_\rho \times E_\rho$, with the Bernstein ellipse $E_\rho = \{ z\in \C:\, |z-1|+|z+1|\le \rho+\rho^{-1}\}$. Then for any $\tilde \rho \in (1,\rho)$ it holds that
\begin{equation} \label{eq:optproblem}
 \inf_{p \in \Pi_{k-1,k-1}} \max_{x,y \in [-1,1]} |f(x,y)-p(x,y)| = \mathcal O(\tilde \rho^{-k}),\quad k\to \infty,
\end{equation}
see, e.g.,~\cite{Trefethen2017}. Hence, Algorithm~\ref{alg:arnoldi} converges linearly as $\ell=k \to\infty$ with a rate arbitrarily close to $\rho$.

For $f(x,y) = 1/(\alpha+x+y)$, a specification of~\eqref{eq:optproblem} can be found in~\cite[Lemma A.1]{Kressner2010}, resulting in a convergence bound for Sylvester equation that matches the asymptotics of~\cite{Simoncini2009a}. This is also an example for a function of the form $f(x,y) = g(x+y)$ with a univariate function $g$. By choosing an approximating polynomial of the same form, the convergence estimate of Theorem~\ref{theorem:convergence} simplifies for any such function to
\begin{equation} \label{eq:simplifiedresult}
 \|f\{A,B\}(cd^T) - U_k X_{k,\ell} V_\ell^T \|_F \le 2M \|c\|_2 \|d\|_2 \cdot \min_{p \in \Pi_{k-1}} \max_{z \in \mathbb E_A \oplus \mathbb E_B} |g(z)-p(z)|,
\end{equation}
where $\oplus$ denotes the Minkowski sum of two sets and $\Pi_{k-1}$ is the set of all (univariate) polynomials of degree at most $k-1$.

We now use~\eqref{eq:simplifiedresult} to analyze the Krylov subspace method for the time-limited Gramian~\eqref{eq:timelimitedeqn} for a symmetric negative definite matrix $A$ with eigenvalues contained in the interval $[-\beta,-\alpha]$, $0<\alpha < \beta <\infty$. By combining~\eqref{eq:timelimited} and~\eqref{eq:simplifiedresult}, convergence estimates can be obtained by studying the polynomial approximation of $g(z) = z^{-1} (\exp(t_ez) - \exp(t_s z))$ on the interval $[-\beta,-\alpha]$. For $t_e = \infty$, $g$ always has a singularity at $z = 0$. In turn, the asymptotic linear convergence rate $\rho$ predicted by polynomial approximation is independent of $t_s\ge 0$. In other words, for $t_e = \infty$ the convergence behavior for time-limited Gramians ($t_s>0$) and Lyapunov equations ($t_s = 0$) are expected to be similar. for $t_e < \infty$, the situation is dramatically different: $g$ is an entire function, yielding superlinear convergence. For $t_s = 0$, $g(z) = z^{-1} (\exp(t_ez) - 1) = t_e \varphi(t_ez)$ and Lemma~\ref{lemma:approxphi} in the appendix can be applied to obtain quantitative convergence estimates.
\begin{example} \label{example:timelimited}
To illustrate the convergence of Algorithm~\ref{alg:arnoldi} for approximating time-limited Gramians, we consider a $500\times 500$ diagonal matrix $A$ with eigenvalues uniformly distributed in $[-100,-0.1]$ and a random vector $c$ of norm 1. Figure~\ref{fig:convergencetimelimited} reports the error $\|X-\tilde X_k\|_2$ (vs. $k$) of the approximation $\tilde X_k$ returned by Algorithm~\ref{alg:arnoldi} with $\ell = k$. The left plot displays the effect of varying $t_s$ while keeping $t_e = \infty$ fixed. 
\begin{figure}
\begin{minipage}{.49\textwidth}
\includegraphics[width=\textwidth]{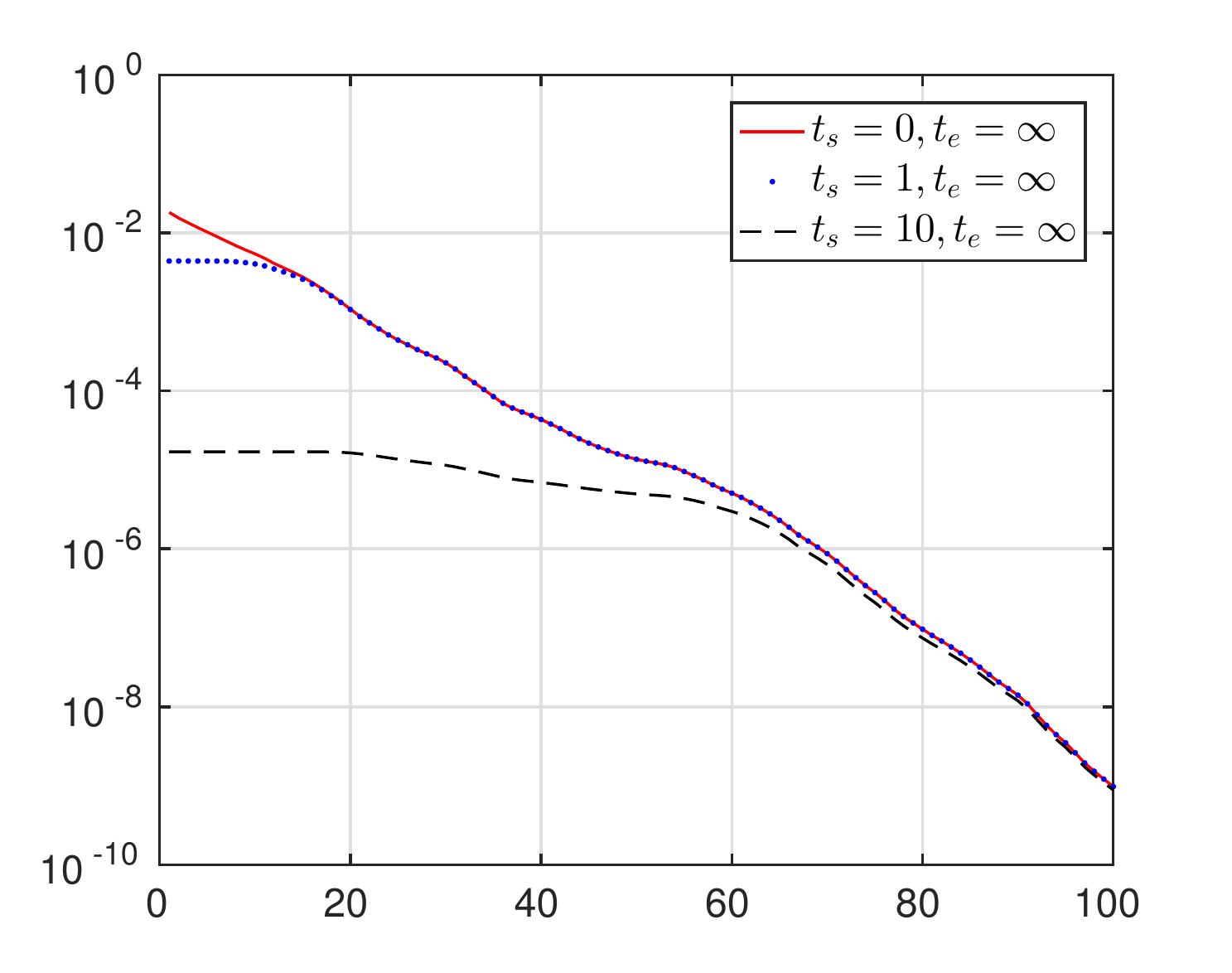}
\end{minipage}
\begin{minipage}{.49\textwidth}
\includegraphics[width=\textwidth]{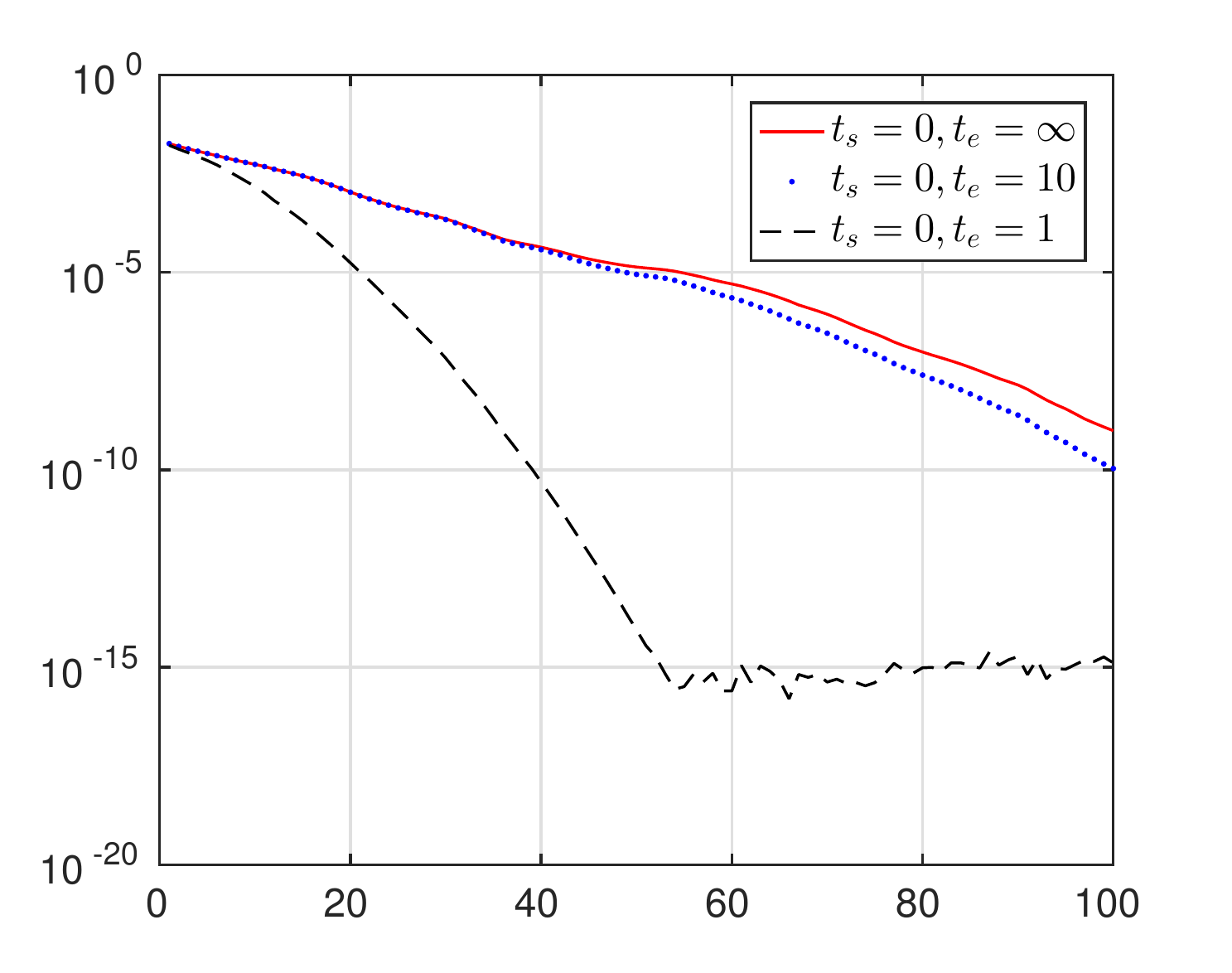}
\end{minipage}
\caption{\label{fig:convergencetimelimited} Convergence of Algorithm~\ref{alg:arnoldi} applied to the time-limited Gramians from Example~\ref{example:timelimited} for different choices of $t_s,t_e$.}
\end{figure}
While there is a pronounced difference initially, probably due to the different norms of $X$, the convergence eventually settles at the same curve.
The right plot displays the effect of choosing $t_e$ finite, clearly exhibiting superlinear convergence for $t_e = 1$.
\end{example}

\section{Application to Fr\'echet derivatives} \label{sec:frechet}

Given a univariate function $f$ analytic in a neighborhood of the eigenvalues of $A$, the Fr\'echet derivative of $f$ at $A$ is a linear map $Df\{A\}\!: \mathbb C^{n\times n} \to \mathbb C^{n\times n}$ uniquely defined by the property $f(A+E) = f(A) + Df\{A\}(E) + \mathcal O(\|E\|_2^2)$. In~\cite[Thm 5.1]{Kressner2014bivariate} it was shown that $Df\{A\} = f^{[1]}\{A,A^T\}$ for the function $f^{[1]}$ defined in~\eqref{eq:finitedifference}. In turn, this enables us to use Algorithm~\ref{alg:arnoldi} for approximating the application of $Df\{A\}$ to rank-one or, more generally, to low-rank matrices. This may be, for example, of interest when approximating gradients in the solution of optimization problems that involve matrix functions; see~\cite{Thanou2017} for an example.

When applying Algorithm~\ref{alg:arnoldi} to $f^{[1]}\{A,A^T\}$ with $\ell = k$, the reduced problem
$f^{[1]}\{G_k,H_k \}$ does, in general, not satisfy $H_k = G_k^T$ and can therefore not be related to a Fr\'echet derivative of $f$ (unless $A$ is Hermitian and $d$ is a scalar multiple of $\overline c$). The following lemma shows that a well-known formula for the Fr\'echet derivative (see, e.g.,~\cite[Thm. 2.1]{Mathias1996}) carries over to this situation.
\begin{lemma} \label{lemma:frechet} Let $f$ be analytic on a domain $\Omega$ containing the eigenvalues of $A \in \C^{m\times m}$ and $B \in \C^{n\times n}$. Then
\[
 f\left( \begin{bmatrix}
A & C \\
0 & B
         \end{bmatrix}
\right) = \begin{bmatrix}
f(A) & f^{[1]}\{A,B^T\}(C) \\
0 & f(B)
         \end{bmatrix}.
\]
\end{lemma}
\begin{proof}
The assumption of the lemma implies that Assumption~\ref{assumption:holo} is satisfied for $f^{[1]}\{A,B^T\}$ with domains $\Omega_A,\Omega_B$ satisfying
$\overline{\Omega_A\cup\Omega_B}\subset \Omega$. Let $\Gamma \subset \Omega$ be a closed contour enclosing $\Omega_A$ and $\Omega_B$. Combining the contour integral representation~\eqref{eq:contour} with 
\[
 f^{[1]}(x,y) = \frac{1}{2\pi\mathrm{i}} \oint_{\Gamma} \frac{f(z)}{(z-x)(z-y)} \,\text{d}z, \qquad \forall x \in \Omega_A, y \in \Omega_B,
\]
gives
\begin{eqnarray*}
f^{[1]}\{A,B^T\}(C) &=& -\frac{1}{8\pi^3 \mathrm{i}} \oint_{\Gamma_A} \oint_{\Gamma_B} \Big[ \oint_{\Gamma} \frac{f(z)}{(z-x)(z-y)} \,\text{d}z \Big] (xI-A)^{-1} C (yI-B)^{-1}\,\text{d}y\,\text{d}x  \\
&=& -\frac{1}{8\pi^3 \mathrm{i}} \oint_{\Gamma} f(z) \Big[ \oint_{\Gamma_A} \frac{(xI-A)^{-1}}{z-x} \text{d}x \Big] C \Big[ \oint_{\Gamma_B} \frac{(yI-B)^{-1}}{z-y} \text{d}y \Big]\,\text{d}z \\
&=& \frac{1}{2\pi\mathrm{i}} \oint_{\Gamma} f(z) (zI-A)^{-1} C (zI-B)^{-1}\,\text{d}z.
\end{eqnarray*}
On the other hand, we have
\begin{eqnarray*}
 f\left( \begin{bmatrix}
A & C \\
0 & B
         \end{bmatrix}
\right) &=& \frac{1}{2\pi\mathrm{i}} \oint_{\Gamma} f(z)  \begin{bmatrix}
zI-A & -C \\
0 & zI-B
         \end{bmatrix}^{-1}\,\text{d}z \\
&=& \begin{bmatrix}
f(A) & \frac{1}{2\pi\mathrm{i}} \oint_{\Gamma} f(z) (zI-A)^{-1} C (zI-B)^{-1}\,\text{d}z  \\
0 & f(B)
         \end{bmatrix},
\end{eqnarray*} 
which completes the proof.
\end{proof}

\begin{algorithm}
\caption{Arnoldi method for approximating $Df\{A\}(cd^T)$}
\label{alg:arnoldifrechet}
\begin{algorithmic}[1]
\State Perform $k$ steps of the Arnoldi method to compute an orthonormal basis $U_k$ of $\calK_k(A,c)$ and $G_k = U_k^\ast A U_k$, $\tilde c = U_k^* c$.

\State Perform $k$ steps of the Arnoldi method to compute an orthonormal basis $V_k$ of $\calK_k(A^T,d)$ and $H_k = V_k^\ast A^T V_k$, $\tilde d = V_k^* d$.
\State Compute $F = f\left( \begin{bmatrix}
G_k & \tilde c \tilde d^T \\
0 & H_k^T
         \end{bmatrix} \right)$ and set $X_k = F(1:k,k+1:2k)$.

\State Return $U_k X_{k} V_k^T$.
\end{algorithmic}
\end{algorithm}

When applying Algorithm~\ref{alg:arnoldi} to $f^{[1]}$, we can now use Lemma~\ref{lemma:frechet} to address the reduced problem with a standard method for evaluating small and dense matrix functions. This yields Algorithm~\ref{alg:arnoldifrechet}.

The following convergence result is a consequence of Theorem~\ref{theorem:convergence}; the particular structure of $f^{[1]}$ allows us to reduce the bivariate to a univariate polynomial approximation problem. 
\begin{corollary} \label{corollary:finite}
Let $f$ be analytic on a domain $\Omega_A$ containing $\calW(A)$ and let $\mathbb E_A$ be a compact convex set such that
$\calW(A) \subset \mathbb E_A \subset \Omega_A$. Then the output of Algorithm~\ref{alg:arnoldifrechet} satisfies the error bound
\[
\|Df\{A\}(cd^T) - U_k X_{k} V_k^T \|_F \le 2 M \|c\|_2 \|d\|_2 \min_{p \in \Pi_{k-1}} \max_{x \in \mathbb E_A} |f^\prime(x)-p(x)|,
\]
where $M = 1$ if $A$ is normal and $M =\frac{1+\sqrt{2}}{2\pi} \frac{\mathrm{len}(\partial \mathbb E_A)}{d_A(\partial \mathbb E_A)}$ otherwise.
\end{corollary}
\begin{proof}
 The conditions of the corollary imply that the conditions of Theorem~\ref{theorem:convergence} are satisfied for $f^{[1]}\{A,A^T\}$, which in turn yields
%\begin{equation} \label{eq:bivariateapprox}
\[
\|f^{[1]}\{A,A^T\}(cd^T) - U_k X_{k} V_k^T \|_F \le 2M \|c\|_2 \|d\|_2 \cdot \inf_{p \in \Pi_{k-1,k-1}} \max_{x,y \in \mathbb E_A} |f^{[1]}(x,y)-p(x,y)|.
\]
For arbitrary $q \in \Pi_{k}$, we let $\tilde p(x,y) := q^{[1]}(x,y) \in \Pi_{k-1,k-1}$ and set $
e:=f-q
$.
By the mean value theorem and convexity of $\mathbb E_A$, for every $x,y\in \mathbb E_A$ with $x\not=y$ there is $\xi \in \mathbb E_A$ such that
\[
e^\prime(\xi) = \frac{e(x)-e(y)}{x-y} = f^{[1]}(x,y) - \tilde p(x,y).
\]
Hence,
\[
 \max_{x,y \in \mathbb E_A} |f^{[1]}(x,y)-\tilde p(x,y)| \le \max_{\xi \in \mathbb E_A} |e^\prime(\xi)| = \max_{\xi \in \mathbb E_A} |f^\prime(\xi)-q^\prime(\xi)|.
\]
Setting $p = q^\prime \in \Pi_{k-1}$ completes the proof.
\end{proof}

Corollary~\ref{corollary:finite} indicates that the convergence of Algorithm~\ref{alg:arnoldifrechet} is similar to the convergence of the standard Arnoldi method for approximating $f(A)c$ and $f(A^T)d$.
Moreover, Corollary~\ref{corollary:finite} allows us to directly apply existing polynomial approximation results derived for studying the convergence of the latter method, such as the ones from~\cite{Beckermann2009,Hochbruck1997}.
\begin{example} \label{example:frechet}
We consider the matrix $A$ and the vector $c$ from Example~\ref{example:timelimited} and measure the error $\|Df\{A\}(cc^T)-F_k\|_2$ of the approximation $F_k = U_kX_kU_k^T$ 
returned by Algorithm~\ref{alg:arnoldifrechet}. This is compared with the error $\|f^\prime(A)c - U_k f^\prime(G_k) \tilde d\|_2$ of the standard Arnoldi approximation for $f^\prime(A)c$.
\begin{figure}
\begin{minipage}{.49\textwidth}
\includegraphics[width=\textwidth]{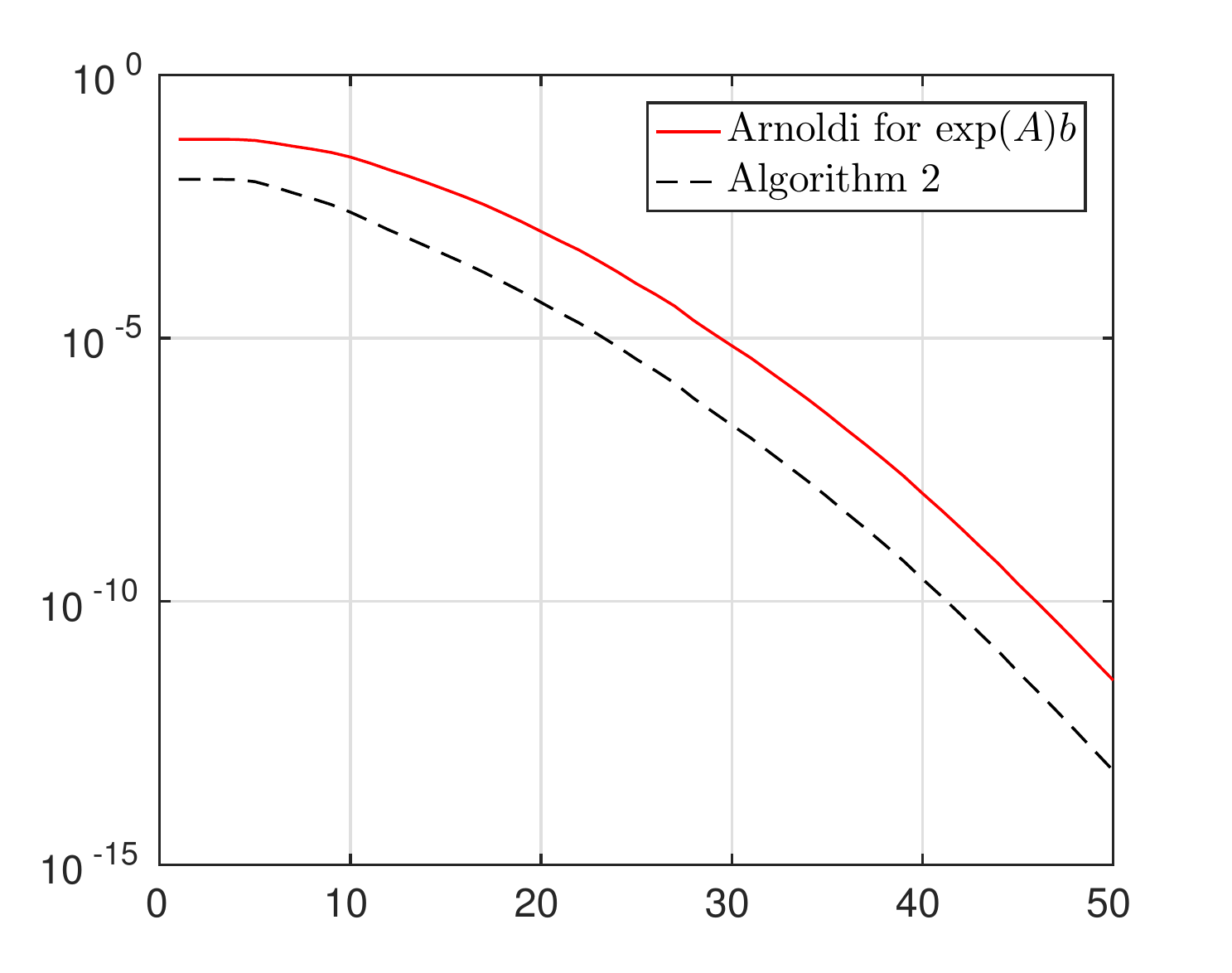}
\end{minipage}
\begin{minipage}{.49\textwidth}
\includegraphics[width=\textwidth]{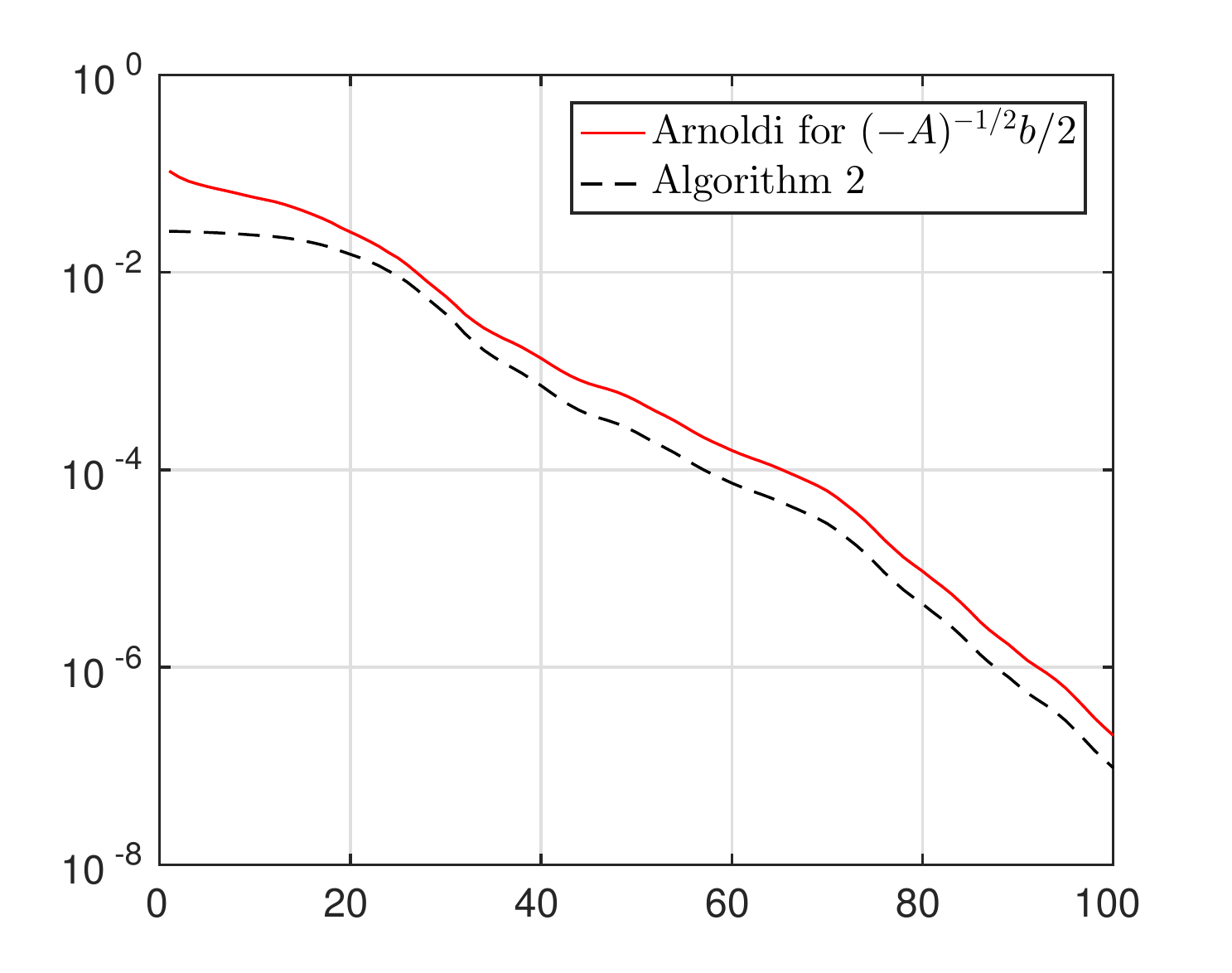}
\end{minipage}
\caption{\label{fig:convergencefrechet} Convergence of Algorithm~\ref{alg:arnoldifrechet} for approximating $Df\{A\}(cc^T)$ and convergence of Arnoldi method for approximating $f^\prime(A)c$ for $f(z) = \exp(z)$ (left plot) and $f(z) = \sqrt{-z}$ (right plot).}
\end{figure}
Figure~\ref{fig:convergencefrechet} demonstrates that both algorithms exhibit the same qualitative convergence behavior.
\end{example}

\section{Outlook}

This work offers numerous opportunities for future work. Most notably, it remains an open problem whether the result of Lemma~\ref{lemma:normestimate} can be established with a constant independent of $A$. Although it is immediate to extend Algorithm~\ref{alg:arnoldi} to rational Krylov subspaces, the implementation of such an approach, in particular the choice of poles, certainly requires further attention.

%Apart from unifying existing methods, our convergence analysis also provides a number of novel insights. This includes, in particular, the connection between the approximation of time-limited Gramians and the $\varphi$ function from Section ???
%Fr\'echet derivative and derivative.
%These extend for some of the instances mentioned in the introduction, see~\cite{Benner2009,Benner2016limitedgramians,Simoncini2007},

\begin{paragraph}{Acknowledgments.}
The author thanks Marcel Schweitzer for inspiring discussions on the topic of this work and Christian Lubich for the idea of the proof for Lemma~\ref{lemma:approxphi}.
\end{paragraph}

\bibliographystyle{siam}
%\bibliography{matrixfunctions}

\bibliography{anchp}

\appendix

\section{Polynomial approximation of the $\varphi$ function}

The $\varphi$ function, which plays an important role in exponential integrators, is given by $\varphi(z) = (\exp(z)-1)/z$. As $\varphi$ is an entire function, we expect polynomial approximations to converge superlinearly. The following lemma derives such an error bound when considering approximations on an interval $[-4\rho,0]$.
\begin{lemma} \label{lemma:approxphi}
Let $\rho> 0$ and $\varepsilon_k = \min_{p\in \Pi_{k-1}} \max_{z \in [-4\rho,0]} |\varphi(z)-p(z)|$. Then
\begin{eqnarray}
 \varepsilon_k &\le & 40 \frac{\rho^2}{k^3} \exp\left(-\frac{k^2}{5\rho}\right)\qquad \text{for $\sqrt{4\rho} \le k \le 2\rho$}, \label{eq:part1} \\
 \varepsilon_k &\le & \frac{8}{3k-5\rho} \left(\frac{e\rho}{k+2\rho}\right)^k \qquad \text{for $k \ge 2\rho$}. \label{eq:part2} \\
\end{eqnarray}
\end{lemma}
\begin{proof}
 We use $x \mapsto (2x-2)\rho$ to map $[-1,1]$ to $[-4\rho,0]$, yielding the equivalent polynomial optimization problem 
 \[
  \varepsilon_k = \min_{p\in \Pi_{k-1}} \max_{x \in [-1,1]} |\tilde \varphi(x)-p(x)|,
 \]
 with $\tilde \varphi(x) := \varphi((2x-2)\rho)$. By~\cite[Theorem 2.2]{Lubich2008}, we have for any $r>1$ that
 \[
  \varepsilon_k \le 2 \mu(\tilde \varphi,r) \frac{r^{-k}}{1-r^{-1}},
 \]
 where
 \begin{eqnarray*}
    \mu(\tilde \varphi,r) & \le & \max_{w\in \mathbb C \atop |w| = r} \left|\tilde \varphi\left(  \left(w+w^{-1} \right)/2 \right) \right| = 
    \max_{w\in \mathbb C \atop |w| = r} \left|\varphi\left( \left(w+w^{-1} -2\right) \rho \right) \right| \\ &=& \left|\varphi\left( \left(r+r^{-1} -2\right) \rho \right) \right| 
    \le  \frac{\exp\left(\left(r+r^{-1} -2\right) \rho \right) }{\left(r+r^{-1} -2\right) \rho}.
 \end{eqnarray*}
The expression $\exp((r+r^{-1} -2)\rho) r^{-k}$ is minimized by setting $r := \frac{k}{2\rho} + \sqrt{\frac{k^2}{4\rho^2}+1}$. Note that $r^{-1} = \sqrt{\frac{k^2}{4\rho^2}+1}-\frac{k}{2\rho}$ and $(r+r^{-1} -2)\rho = \sqrt{k^2+4\rho^2}-2\rho$.

We first discuss the case $\sqrt{4\rho} \le k \le 2\rho$, which in particular implies $\rho\ge 1$.
The inequality 
\begin{equation} \label{eq:inequality}
 \frac{\sqrt{k^2+4\rho^2} - 2\rho}{k} + \log\left( \sqrt{\frac{k^2}{4\rho^2}+1}-\frac{k}{2\rho} \right) \le -\frac{k}{5\rho}
\end{equation}
is shown for $k = \sqrt{4\rho}$ by direct calculation. By differentiating, it is shown that the difference between both sides of~\eqref{eq:inequality} is montonically decreasing for $k \in [\sqrt{4\rho},2\rho]$ and hence the inequality holds for all such $k$.
% rho = 2000; k = linspace(2*sqrt(rho),2*rho,1000);
% plot(k,(sqrt(k.^2+4*rho^2)-2*rho)./k+log(sqrt(k.^2/4/rho^2+1)-k/2/rho)-k/5/rho,'r-'),
Using also
\[
 \sqrt{k^2+4\rho^2}-2\rho \ge \frac{k^2}{5\rho}, \qquad 1-r^{-1} \ge \frac{k}{4\rho},
\]
we obtain from~\eqref{eq:inequality} that 
 \[
\mu(\tilde \varphi,r) \frac{r^{-k}}{1-r^{-1}} \le \frac{\exp(\sqrt{k^2+4\rho^2}-2\rho)}{\sqrt{k^2+4\rho^2}-2\rho} \cdot \frac{\exp\left(k\log\left(\sqrt{\frac{k^2}{4\rho^2}+1}-\frac{k}{2\rho}\right)\right)}{1-r^{-1}} 
\le 20 \frac{\rho^2}{k^3} \exp\left(-\frac{k^2}{5\rho}\right),
\]
which completes the proof of~\eqref{eq:part1}.

Similarly, the inequality~\eqref{eq:part2} follows from combining
\[
 \frac{\sqrt{k^2+4\rho^2} - 2\rho}{k} + \log\left( \sqrt{\frac{k^2}{4\rho^2}+1}-\frac{k}{2\rho} \right) \le \log(e\rho)-\log(k+2\rho)
\]
with
\[
 (\sqrt{k^2+4\rho^2}-2\rho)(1-r^{-1}) \ge \frac34 k-\frac54 \rho,
\]
% rho = 10; k = linspace(2*rho,8*rho,1000);
% plot(k,(sqrt(k.^2+4*rho^2)-2*rho).*(1-sqrt(k.^2/4/rho^2+1)+k/2/rho),'b-'), hold on, plot(k,0.75*k-1.25*rho,'r-'), hold off
which hold for $k\ge 2\rho$.
\end{proof}

Compared to the corresponding bounds for the exponential~\cite[Theorem 2]{Hochbruck1997}, the bounds of Lemma~\ref{lemma:approxphi} are lower for larger $k$, primarily because they benefit from the additional factor $\mathcal O(1/k)$ due to the slower growth of the $\varphi$ function. Additionally, the factor $\big(\frac{e\rho}{k+2\rho}\big)^k$ in~\eqref{eq:part2} seems to be better than the corresponding factor $\exp(-\rho) \big(\frac{e\rho}{k}\big)^k$ ~\cite[Eqn. (14)]{Hochbruck1997}. This improvement can probably be carried over to the exponential. Figure~\ref{fig:convphi} illustrates the differences between the bounds.
\begin{figure}
\begin{minipage}{.49\textwidth}
\includegraphics[width=\textwidth]{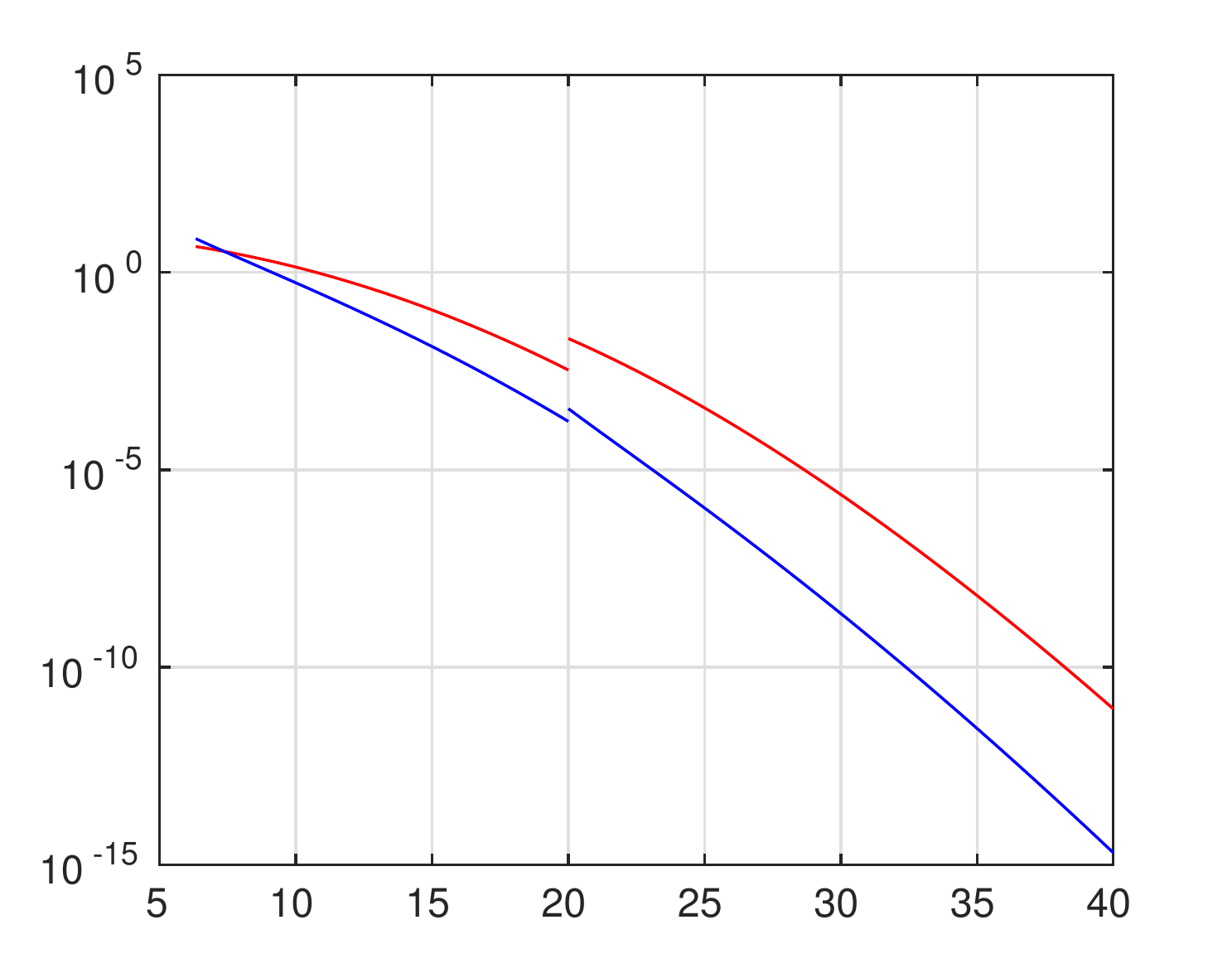}
\end{minipage}
\begin{minipage}{.49\textwidth}
\includegraphics[width=\textwidth]{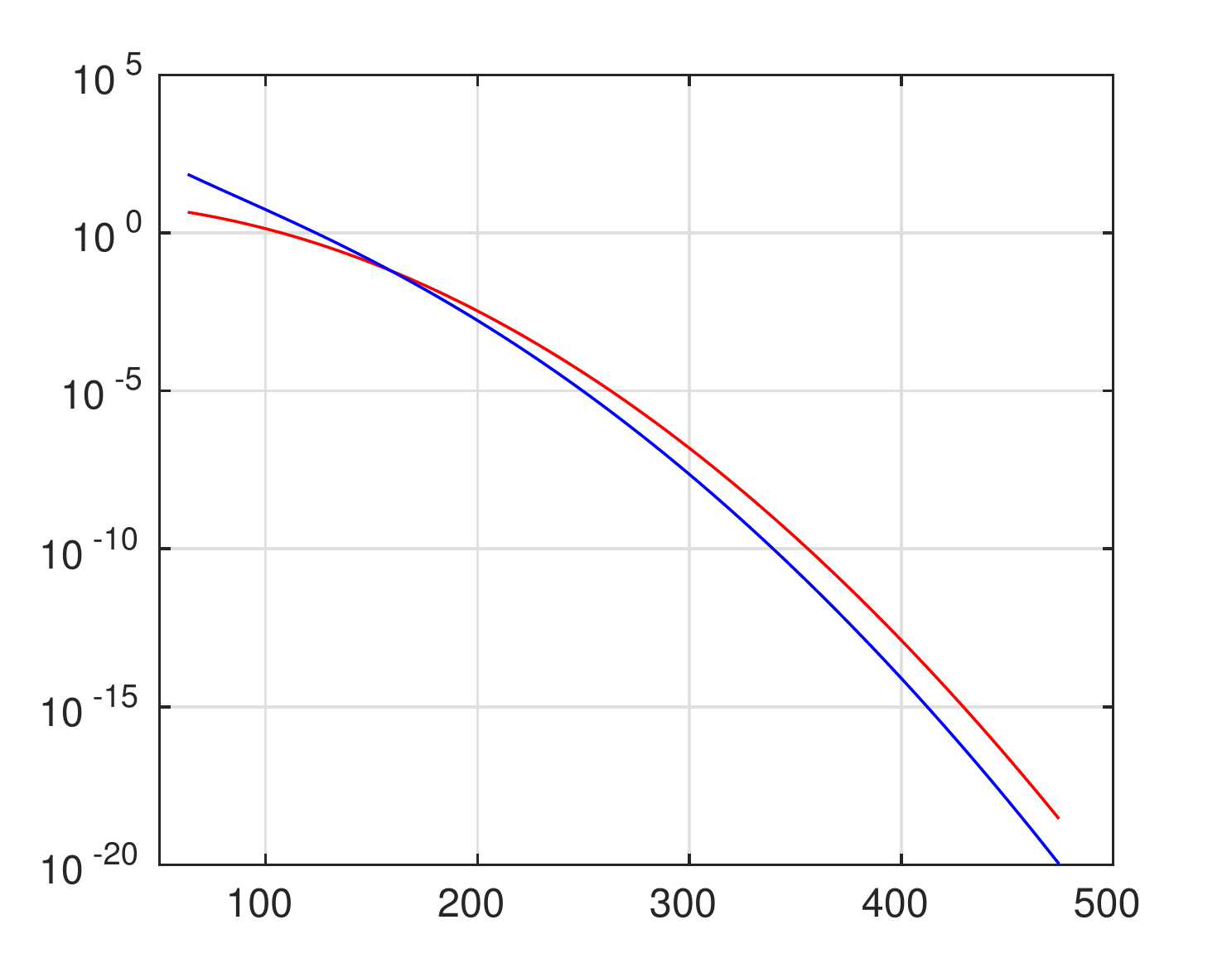}
\end{minipage}
\caption{\label{fig:convphi} Bounds of Lemma~\ref{lemma:approxphi} for the polynomial approximation of the $\varphi$ function (in blue) and bounds of~\cite[Theorem 2]{Hochbruck1997} for the polynomial approximation of the exponential function (in red). Left plot: $\rho = 10$. Right plot: $\rho = 1000$.}
\end{figure}

\end{document}